\documentclass{alggeom}
\usepackage{amsmath}
\usepackage{multirow}
\usepackage[all]{xy}
\usepackage{ifthen}
\newcommand{\cit}[1]{{\rm \textbf{#1}}}

\newcommand{\Ref}[2]{\cit{%
\ifthenelse{\equal{#1}{thm}}{Theorem}{}%
\ifthenelse{\equal{#1}{prop}}{Proposition}{}%
\ifthenelse{\equal{#1}{lem}}{Lemma}{}%
\ifthenelse{\equal{#1}{cor}}{Corollary}{}%
\ifthenelse{\equal{#1}{defn}}{Definition}{}%
\ifthenelse{\equal{#1}{oss}}{Remark}{}%
\ifthenelse{\equal{#1}{rmk}}{Remark}{}%
\ifthenelse{\equal{#1}{sec}}{Section}{}%
\ifthenelse{\equal{#1}{ex}}{Example}{}%
\ifthenelse{\equal{#1}{conj}}{Conjecture}{}%
\ifthenelse{\equal{#1}{ssec}}{Subsection}{}%
\ifthenelse{\equal{#1}{tab}}{Table}{}%
\ifthenelse{\equal{#1}{cla}}{Claim}{}%
\  \ref{#1:#2}%
}}
\newtheorem{prop}{Proposition}[section]
\newtheorem*{prop*}{Proposition}
\newtheorem{thm}[prop]{Theorem}
\newtheorem{lem}[prop]{Lemma} 
\newtheorem{cor}[prop]{Corollary}

\theoremstyle{definition}
\newtheorem{defn}[prop]{Definition}
\newtheorem{ex}[prop]{Example}

\theoremstyle{remark}
\newtheorem{oss}[prop]{Remark}
\newtheorem{rmk}[prop]{Remark}

\numberwithin{equation}{section}
\newcommand{\hk}{hyperk\"{a}hler }

\newcommand{\kahl}{K\"{a}hler }

\newcommand{\mc}[1]{\mathcal{#1}}
\DeclareMathOperator*{\rk}{rk}
\DeclareMathOperator*{\Pic}{Pic}
\newcommand{\NS}{\mathrm{NS}}
\newcommand{\ZZ}{\mathbb{Z}}
\newcommand{\Z}{\mathbb{Z}}
\newcommand{\C}{\mathbb{C}}
\newcommand{\Q}{\mathbb{Q}}
\newcommand{\R}{\mathbb{R}}

\newcommand{\PP}{\mathbb{P}}
\newcommand{\CC}{\mathbb{C}}

\newcommand{\aut}{\mathrm{Aut}}
\newcommand{\id}{\mathrm{id}}
\newcommand{\Hom}{\mathrm{Hom}}
\newcommand{\Aut}{\mathrm{Aut}}
\newcommand{\GL}{\mathrm{GL}}
\renewcommand{\mod}[1]{\mathrm{\text{ }(mod\text{ }#1)}}
\newcommand{\lat}[1]{\left\langle#1\right\rangle}

\usepackage{hyperref}

\begin{document}

\title{Prime order automorphisms of abelian surfaces:\\ a lattice-theoretic point of view}
\author{Giovanni Mongardi}
\email{mongardi@mat.uniroma3.it}
\address{Department of Mathematics, University of Milan\\ via Cesare Saldini 50, Milan, Italy}
\author{K\'evin Tari}
\email{kevin.tari@math.univ-poitiers.fr}
\address{Laboratoire de Math\'ematiques et Applications, Universit\'e de Poitiers\\
T\'el\'eport 2 - BP 30179
Boulevard Marie et Pierre Curie\\
86962 Futuroscope Chasseneuil Cedex,
France }
\author{Malte Wandel}
\email{wandel@kurims.kyoto-u.ac.jp}
\address{RIMS, Kyoto University\\ Kitashirakawa Oiwake cho, Sakyo-ku,
Kyoto 606-8502, Japan}
%

\thanks{First named author was supported by FIRB 2012 ``Spazi di Moduli e applicazioni'' and partially supported by SFB/TR 45 ``Periods, moduli spaces and arithmetic of algebraic varieties''\\The third named author is supported by JSPS Grant-in-Aid for Scientific Research (S)25220701}

\begin{abstract}
In this article, we focus on a new perspective of automorphisms of complex 2-tori, reviewing previous works from a lattice-theoretic point of view. In particular, we give a classification of families of symplectic and non-symplectic automorphisms of prime order.
\end{abstract}

\maketitle
\tableofcontents

\vspace*{6pt}
\setcounter{section}{-1}
\section{Introduction}
This article grew out of the attempt to study prime order automorphisms of generalised Kummer manifolds, which will be done in the upcoming article \cite{MTW}. These are irreducible holomorphic symplectic manifolds introduced by Beauville (\cite{Bea83}), which can be seen as a generalisation of the classical notion of Kummer surfaces to higher dimensions. In particular, for all $n\geq 3$ there is a natural way to associate to each complex $2$-torus $A$ a generalised Kummer manifold $K_n(A)$ of complex dimension $2n-2$ and every group automorphism of the torus will induce in a natural way an automorphism of the Kummer manifold (a so-called \em natural \em automorphism). The most important invariant of an irreducible holomorphic symplectic manifold $X$ is its second integral cohomology $H^2(X,\mathbb{Z})$ which carries - besides its pure weight two Hodge structure - a non-degenerate symmetric bilinear form. Any automorphism of $X$ acts on $H^2(X,\mathbb{Z})$ by a Hodge isometry. The study of automorphisms of these manifolds is therefore closely related to the study of the underlying lattices and their isometries. In the case of a natural automorphism of $X=K_n(A)$ the action on the lattice $H^2(X,\mathbb{Z})\simeq H^2(A,\Z)\oplus\lat{-2n-2}$ is completely determined by the action of the underlying $2$-torus-automorphism on the $H^2(A,\mathbb{Z})$ of the torus. Thus we are interested in a classification of automorphisms of $2$-tori using the induced action on $H^2(A,\mathbb{Z})$.

On the other hand, when studying tori and their automorphisms, the most important invariant is the Hodge structure on $H^1(A,\mathbb{Z})$. Automorphisms of $2$-tori have been studied by several people with the final classification given by Fujiki in \cite{Fuj88}. For further references see loc.\ cit.\ and the textbook of Birkenhake$-$Lange \cite{BL92}. The classification by Fujiki gives a complete list of pairs $(A,G)$ where $A$ is a complex $2$-torus and $G$ a finite group of group automorphisms of $A$. However, the induced action on $H^2(A,\mathbb{Z})$ is not in the focus of his considerations. (It is implicitly contained in the case of symplectic automorphisms, see \cite[Section 6]{Fuj88}.)
In this article we complement Fujiki's classification by analysing, for each example, the induced action on the second cohomology lattice of the $2$-torus. In particular, we give the invariant lattices of the isometries, restricting ourselves to the case of prime order automorphisms. For the non-symplectic case, the invariant lattice of such an automorphism characterises the Fujiki's family it is contained in. When $\sigma$ is an automorphism of $A$, we denote by $G_\sigma$ the group generated by $\sigma$ and $-\id$. The main result (cf.\ \Ref{prop}{list}) in this case is as follows:
\begin{prop}\label{prop:list2}
Let $A$ be a $2$-dimensional complex torus and $\sigma$ be a non-symplectic group automorphism of $A$ with an action of prime order $p$ on $H^2(A,\Z)$. Then the invariant lattice $T(G_\sigma)$ is isomorphic to one of the lattices contained in the table below. We denote by $r$ the rank of $T(G_\sigma)$, by $a$ the length of its discriminant group and by $\mathrm{dim}$ the dimension of the family of abelian surfaces admitting an embedding $G_\sigma\subseteq\Aut_0(A)$.

\vspace{5pt}
\renewcommand{\arraystretch}{1.5}
\begin{center}
\begin{tabular}{c|c|c|c|c}
$p$&$r$&$\mathrm{dim}$&$a$&$T(G_\sigma)$\\
\hline
\multirow{4}{*}{$2$}&\multirow{3}{*}{$2$}&\multirow{3}{*}{$2$}&$0$&$U$\\
\cline{4-5}
&&&$2$&$U(2)$\\
\cline{4-5}
&&&$2$&$\langle2\rangle\oplus\langle -2\rangle$\\
\cline{2-5}
&$4$&$0$&$2$&$U\oplus\langle -2\rangle^{\oplus2}$\\
\hline
\multirow{3}{*}{$3$}&\multirow{2}{*}{$2$}&\multirow{2}{*}{$1$}&$0$&$U$\\
\cline{4-5}
&&&$2$&$U(3)$\\
\cline{2-5}
&$4$&$0$&$1$&$U\oplus A_2(-1)$\\
\hline
$5$&$2$&$0$&$1$&$H_5$
\end{tabular}
\end{center}

\vspace{5pt}
\end{prop}

We recall that this article is meant to deal as a reference list for the upcoming article \cite{MTW}.

The structure of the article is as follows: We begin by recalling basic facts on tori, their automorphisms and lattices in Section 2. Then we state the Torelli theorem due to Shioda and its consequences for the study of automorphisms of $2$-tori in Section 3. The lattice theoretic classification is performed in the non-symplectic case in Section 4 and in the symplectic case in Section 5. We finish the article by summarising some generalities on moduli spaces of manifolds with automorphisms in the final section.

\section*{Acknowledgements}
We would like to thank the Max Planck Institute in Mathematics, for hosting the three authors when this work was started. We are also grateful to Samuel Boissi\`ere and Bert Van Geemen for their comments and support.
\section{Preliminaries}\label{sec:prelim}

\subsection{Generalities on Tori and Automorphisms}
Let us recall some basic facts about tori, we refer to \cite{BL92} for the details. Let $A$ be a complex torus of dimension 2. The whole cohomology of $A$ is obtained from the first cohomology group via $H^k(A,\Z)\cong\Lambda^k H^1(A,\Z)$. From this, one gets also $\Lambda^2 H^2(A,\Z)\cong H^4(A,\Z)\cong\Z$, which endows $H^2(A,\Z)$ with a lattice structure, i.e. a non-degenerate symmetric bilinear form with integer values. This lattice is known to be isomorphic to $U^{\oplus3}$, where $U$ is the hyperbolic plane, i.e.\ the rank $2$ lattice with intersection form \[ U=\left( \begin{array}{cc} 0 & 1 \\ 1 & 0\end{array}\right).\] We denote by $\langle.,.\rangle$ this bilinear form on $H^2(A,\Z)$, which is of signature $(3,3)$. The group $H^{2,0}(A)$ is 1-dimensional and generated by a nowhere vanishing holomorphic $2$-form $\omega$.

Denote by $\aut_0(A)$ the group of group automorphisms of $A$. These are exactly the automorphisms fixing the origin. There is a natural representation of groups to the orthogonal group of the lattice defined above:
\[\nu\colon\Aut_0(A)\rightarrow O(H^2(A,\Z))\]
whose kernel is $\{\pm\id_A\}$ by Lemma 2 of \cite[Section 1]{Shi78}. We will use this representation to analyse the automorphisms of $A$ from a lattice-theoretic point of view.

We want to study automorphisms of complex tori with a prime order action on $H^2(A,\Z)$. Thus we will study pairs $(A,G)$ where $A$ is a complex torus and $G$ is a subgroup of $\aut_0(A)$ such that $\nu(G)\cong \ZZ/p\ZZ$. It is more convenient to do so than studying only one automorphism at a time, because we will classify the action on $H^2(A,\Z)$ through the invariant lattice (see the following subsection). Moreover, since the action of $-\id_A$ on $H^2(A,\Z)$ is trivial, we will always assume that $-\id_A\in G$. If $\sigma\in\Aut_0(A)$ has an action of order $p$ on $H^2(A,\Z)$, we denote by $G_\sigma$ the subgroup of $\Aut_0(A)$ generated by $\sigma$ and $-\id_A$.

\begin{defn}
We call the group $G_\sigma$ {\em symplectic} if $\sigma$ preserves $\omega$, we call it {\em non-symplectic} otherwise.
\end{defn}

Let us study the possible orders for these actions and for the corresponding automorphisms: If $A=\C^2/\Lambda$ is a $2$-dimensional complex torus, every $\sigma\in\Aut_0(A)$ of finite order comes from an element of $\GL(\C^2)$ which induces an automorphism of $\Lambda$. Now, this can be seen as an element $g$ of $\GL_4(\Z)$ and since $g^d=\id$, the polynomial $X^d-1$ annihilates this matrix and it is diagonalizable over $\C$, with roots of the unity as eigenvalues. The characteristic polynomial is then of the form $(X-1)^k\prod\Phi_{n_i}(X)^{k_i}$, where $\Phi_n$ is the $n$-th cyclotomic polynomial. Looking at the degrees, we see that $k+\sum k_i\varphi(n_i)=4$ where $\varphi$ is Euler's totient function, and this gives a short list of possible orders for $g$: $1$, $2$, $3$, $4$, $5$, $6$, $8$, $10$ or $12$. We restrict ourselves to automorphisms with a prime order action on $H^2(A,\Z)$. This shortens the list to $2$, $3$, $4$, $5$, $6$ or $10$, with an action on $H^2(A,\Z)$ of order $2$, $3$, or $5$.

\subsection{Special sublattices of $H^2(A,\Z)$}
Inside $H^2(A,\Z)$, there are several interesting sublattices to consider. First, the N\'eron-Severi lattice of $A$ is defined by:
\[\NS(A)=\{x\in H^2(A,\Z)\mid\langle x,\omega\rangle =0\}\]
and the transcendental lattice is $T_A=\NS(A)^\perp$. The Picard number of $A$ is $\rho(A):=\rk(\NS(A))$. We will study groups of automorphisms through their action on $H^2(A,\Z)$, first by finding all lattices that can occur as invariant lattices of such an action. We denote by $T(G_\sigma)$ the invariant lattice, i.e.\
\[T(G_\sigma):=\{ x\in H^2(A,\Z)\mid g^*x=x\quad\forall g\in G_\sigma\}\] and $S(G_\sigma):=T(G_\sigma)^\perp$ the orthogonal of the invariant lattice of an automorphism group $G_\sigma$.



\subsection{$p$-elementary lattices}
Let $p$ be a prime number. For a lattice $L$, we define the dual lattice $L^\ast:=\Hom(L,\Z)$, the discriminant group $A_L:=L^\ast/L$ and the discriminant $d_L:=|A_L|$. This group inherits a symmetric bilinear form with values in $\mathbb{Q}/\mathbb{Z}$ ($\mathbb{Q}/2\mathbb{Z}$ for even lattices) from that on $L^\ast$, which is called discriminant form and we denote it by $q_L$. We say that $L$ is unimodular if $L^\ast=L$ (or $A_L=\{0\}$), and $p$-elementary if $A_L\cong(\Z/p\Z)^a$ for some integer $a$ called the {\em length} of $A$ satisfying $a\leq\rk(L)$. The signature of a lattice is the signature of the associated real bilinear form and we say that a lattice is hyperbolic if it has signature $(1,b)$ for some $b>0$. If $L$ is a sublattice of a unimodular lattice, then $A_L\cong A_{L^\perp}$ and $q_{L^{\perp}}=-q_L$. Two lattices are said to be of the same genus if they have the same rank, signature and isometric discriminant forms. 

Since $H^2(A,\Z)\cong U^{\oplus3}$ is unimodular, we have $A_{T(G_\sigma)}\cong A_{S(G_\sigma)}$ for any automorphism $\sigma$. When $\sigma$ has an action of prime order $p$ on $H^2(A,\Z)$, the coinvariant lattice is $p$-elementary, and then so is the invariant lattice.

The following theorems deal with the questions of existence and unicity of $p$-elementary lattices. The first one deals with the case $p=2$ and is Theorem 4.3.2 of \cite{Nik83}. Let $r$ be the rank of a $2$-elementary lattice $L$, then its discriminant group is of the form $(\Z/2\Z)^a$. Let $\delta$ be 0 if $q_L(s,s)$ is integer for any element $s$ of $A_L$, let $\delta$ be 1 otherwise.

\begin{thm}[\cite{Nik83}]\label{thm:2el}
An even hyperbolic $2$-elementary lattice of rank $r$ is uniquely determined by the invariants $(r,a,\delta)$ and exists if and only if the following conditions are satisfied:
$$\left\{
\begin{array}{l l}
a\leq r\\
r\equiv a\mod{2}\\
\text{if }\delta=0\text{, then }r\equiv 2\mod{4}\\
\text{if }a=0\text{, then }\delta=0\\
\text{if }a\leq 1\text{, then }r\equiv 2\pm a\mod{8}\\
\text{if }a=2\text{ and }r\equiv 6\mod{8}\text{, then }\delta=0\\
\text{if }\delta=0\text{ and }a=r\text{, then }r\equiv 2\mod{8}
\end{array}
\right.$$
\end{thm}

The second theorem deals with the case $p\neq 2$:

\begin{thm}[{\cite[Section 1]{RS81}}]\label{thm:pel}
An even hyperbolic $p$-elementary lattice of rank $r$ with $p\neq 2$ with invariants $(r,a)$ exists if and only if the following conditions are satisfied:
$$\left\{
\begin{array}{l l}
a\leq r\\
r\equiv 0\mod{2}\\
\text{if }a\equiv 0\mod{2}\text{, then }r\equiv 2\mod{4}\\
\text{if }a\equiv 1\mod{2}\text{, then }p\equiv (-1)^{r/2-1}\mod{4}\\
\text{if }r\not\equiv 2\mod{8}\text{, then }r>a>0
\end{array}
\right.$$
Such a lattice is uniquely determined by the invariants $(r,a)$ if $r\geq 3$.
\end{thm}
The question of unicity of lattices is refined by the following proposition, which is part of Corollary 22 in section 15.9.7 of \cite{CS99}:

\begin{prop}\label{prop:CS99}
An indefinite lattice of rank $2$ and discriminant $d$ with more than one class in its genus has $d\geq 17$.
\end{prop}
Here the genus of a lattice is the set of isometry classes of even lattices with the same signature and discriminant form.

Finally, in order to narrow the list of possible coinvariant lattices, we will use the following proposition, that is proven in section 4 of \cite{BCMS14}:

\begin{prop}\label{prop:BCMS14}
Let $L$ be a lattice with a non-trivial action of order $p$, with rank $p-1$ and discriminant $d_L$, then $\frac{d_L}{p^{p-2}}$ is a square in $\Q$.
\end{prop}

We end this preliminary section with introducing two special lattices that will occur later. Firstly we denote by $A_2$ the unique rank $2$, positive definite, $3$-elementary lattice with $a=1$, i.e.\
\[ A_2:=\left( \begin{array}{cc} 2 & 1 \\ 1 & 2\end{array}\right).\]

Secondly we denote by $H_5$ the lattice
\[H_5:=\left( \begin{array}{cc} 2 & 1 \\ 1 & -2\end{array}\right),\]
which is the unique hyperbolic, rank $2$, $5$-elementary lattice with $a=1$.
\section{The Torelli Theorem by Shioda}
In the paper \cite{Shi78} Shioda studied the periods of complex 2-tori. Here we will review his results and deduce some basic results about families of automorphisms.

Let $A$ be a complex 2-torus. Let $\sigma$ be an automorphism of $A$. It induces an action on the integral cohomologies:
\[\psi:=\sigma^*\in O(H^1(A,\Z)), \quad \varphi:=\Lambda^2\psi\in O(H^2(A,\Z)).\]
Note that, in fact, $\varphi$ is an element of the subgroup of Hodge isometries and its determinant is $1$. Furthermore the cone $\mc{C}$ of positive classes in $H^2(A,\R)$ comes with a distinguished orientation, which is the choice of a connected component, (cf.\ \cite[Section 4]{Mar11}) and $\varphi$ above preserves this orientation. Conversely we have:

\begin{thm}[{\cite[Theorem 1]{Shi78}}]\label{thm:shio}
Let $\varphi\in O(H^2(A,\Z))$ be an orientation preserving Hodge isometry of determinant $1$, then there exists an automorphism $\sigma$ of $A$ such that $\varphi=\Lambda^2\sigma^*$.
\end{thm}

In the language of \hk manifolds, we call monodromy operator any element of $O(H^2(A,\Z))$ which can be obtained by parallel transport along a family of smooth deformations of $A$. The content of the above Torelli theorem can be rephrased as follows.

\begin{thm}\label{thm:shio2}
Let $A$ be a complex 2-torus. Then the group of monodromy operators coincides with the group of orientation preserving isometries. The group of monodromy Hodge isometries is the image of the representation $\nu\colon Aut_0(A)\rightarrow O(H^2(A,\mathbb{Z}))$.
\end{thm}

To understand the significance of the condition on the determinant (being equal to $1$), let us introduce the following notions:

\begin{defn}
A \em marked 2-torus \em $(A,\phi)$ is a pair consisting of a torus $A$ and an isometry $\phi\colon H^2(A,\ZZ)\rightarrow U^3$. An \em isomorphism of marked 2-tori \em is an isomorphism of the underlying tori which respects the markings. We denote the resulting \em moduli space of marked 2-tori \em by $\mc{M}$.

The \em period domain \em for 2-tori is defined as
\[ \Omega=\{[x]\in \mathbb{P}(U^3\otimes \mathbb{C})\mid \langle x,x\rangle =0,\enskip \langle x,\bar{x}\rangle >0\}.\]

For every marked 2-torus $(A,\phi)$ -- if $\omega$ is a holomorphic two form -- the span $[\phi(\omega)]$ is in $\Omega$ and is called the \em period of \em $(A,\phi)$. We obtain thus the \em period map \em $\mc{M}\rightarrow \Omega$.
\end{defn}

Now, the crucial difference to the case of K3 surfaces is the following: For any 2-torus $A$, denote its dual torus by $\hat{A}$. Shioda proved that there exists a canonically defined Hodge isometry
\[\varphi_0\colon H^2(A,\Z)\rightarrow H^2(\hat{A},\Z)\]
of determinant $-1$, which is not a monodromy operator and thus the moduli space $\mc{M}$ of marked 2-tori has two connected components both mapping injectively onto the period domain: This is the second important result of Shioda's article. (Note that we use the language of marked tori instead of isomorphism classes here.)

\begin{thm}[{\cite[Theorem 2]{Shi78}}]\label{thm:shio3}
The period map establishes a surjective two-to-one map from the moduli space of marked 2-tori onto the period domain $\Omega$.
\end{thm}

\begin{oss}
To classify automorphisms of 2-tori, one has to be careful: Assume we are given a family $\mc{A}\rightarrow B$ of tori together with an automorphism $\Sigma$ acting fibrewise. We can form the dual family in the following sense: As manifold we take the relative moduli space of degree zero line bundles $\hat{\mc{A}}:=\Pic^0_B(\mc{A})$ which comes together with an induced automorphism $\hat{\Sigma}$ which is nothing but pullback along $\Sigma$. In order to classify families of automorphisms, we can use \Ref{thm}{shio} to infer the following: Let $\varphi$ be an isometry of $U^{\oplus3}$ of determinant $1$ whose conjugacy class occurs as an induced action of an automorphism of a 2-torus, then there are at most two families of pairs $(A,\sigma)$ of 2-tori $A$ with automorphism $\sigma$ such that the induced action of $\sigma$ on $H^2(A,\Z)$ is conjugate to $\varphi$. The two families are dual to each other in the above sense. If every element in these families is principally polarised, then the two families are actually identical, i.e.\ they contain the same set of pairs $(A,G)$ (cf.\ \Ref{ssec}{duality}). If there are non-principally polarised tori involved, the two families will, in general, be distinct.
\end{oss}

To conclude this section we deduce from the Torelli theorem the following important result on the structure of the invariant lattice:

\begin{cor}\label{cor:evenrank}
For any automorphism $\sigma$ of $A$ with a prime order action on $H^2(A,\Z)$, the rank of the invariant lattice $T(G_\sigma)\subseteq H^2(A,\Z)$ is even.
\end{cor}

\begin{proof}
Let $p$ be the order of $\nu(\sigma)$, we know that $T(G_\sigma)$ is $p$-elementary. In the case $p\neq 2$, \Ref{thm}{pel} directly implies that $\rk(T(G_\sigma))$ is even. Let $d$ be the order of $\sigma$. If $p=2$, we need to consider the two subcases $d=2$ and $d=4$.

If $d=2$, using \Ref{thm}{shio} and splitting $H^1(A,\R)$ in eigenspaces $E_1\oplus E_{-1}$ for $\sigma$, we get
\begin{align*}
r:=&\rk(T(G_\sigma))\\
  =&\dim(H^2(A,\R)^\sigma)\\
  =&\dim(\wedge^2E_1\oplus\wedge^2E_{-1})\\
  =&\binom{k}{2}+\binom{4-k}{2}
\end{align*}
where $k:=\dim(E_1)=\dim(H^1(A,\R)^{\sigma})=\dim(H_1(A,\R)^{\sigma})$. But from a real point of view $A=H_1(A,\R)/H_1(A,\Z)$, so $\dim(H_1(A,\R)^\sigma)$ is two times the complex dimension of the fixed locus of $\sigma$ in $A$. Hence $k$ is even and so is $r$.

If $d=4$, the cohomology splits as before in $H^1(A,\C)=E_1\oplus E_{-1}\oplus E_i\oplus E_{-i}$. The action on $H^1(A,\C)$ comes from the action on $H^1(A,\Z)$, so $\dim(E_i)=\dim(E_{-i})$ and this dimension is non zero since the automorphism has order $4$ on $H^1$. Hence $E_{\pm 1}={0}$, otherwise the action of $\sigma$ on $E_{\pm 1}\wedge E_{\pm i}\subset H^2(A,\C)$ would be of order $4>p$. Therefore we finally get $\dim(E_i)=\dim(E_{-i})=2$ and $r=\dim(E_i\wedge E_{-i})=4$.
\end{proof}
\hspace{1pt}
\section{Non-symplectic automorphisms of prime order}\label{sec:nonsymp}
In this section we give a classification of families of prime order non-symplectic automorphisms of abelian surfaces.

Let $A$ be a complex torus and $G_\sigma$ be a non-sympletic automorphism group with $\sigma$ having an action of prime order $p$ on $H^2(A,\Z)$. Then $\sigma^\ast(\omega)=\zeta_p\omega$, where $\zeta_p$ is a $p$-th root of unity. Since $\sigma$ acts on $H^2(A,\Z)$ as an isometry, for all $x$ in $T(G_\sigma)$ we have:
\[\langle x,\omega\rangle=\zeta_p\langle x,\omega\rangle=0.\]
Hence, $x\in\NS(A)$ and thus $T(G_\sigma)\subset\NS(A)$. Besides, since it admits a non-symplectic automorphism, $A$ is always algebraic by \cite[Proposition 6]{Bea83}. (This is proven by constructing a rational K\"ahler class, which is invariant under the action of $\sigma$). We deduce that $\NS(A)$ has signature $(1,\rho(A)-1)$ by the Hodge index theorem, and $T(G_\sigma)$ must have signature $(1,r-1)$ (where $r$ is the rank of $T(G_\sigma)$) since it contains a multiple of the invariant rational \kahl class. 

We proceed by classifying all possible lattices that may occur as invariant lattices of non-symplectic automorphism groups of prime order.

\begin{prop}\label{prop:list}
Let $A$ be a $2$-dimensional complex torus and $\sigma$ be a non-symplectic group automorphism of $A$ with an action of prime order $p$ on $H^2(A,\Z)$. Then $T(G_\sigma)$ is isomorphic to one of the lattices contained in the table in \Ref{prop}{list2} in the introduction. We denote by $r$ the rank of $T(G_\sigma)$, by $a$ the length of its discriminant group and by $\mathrm{dim}$ the dimension of the family of abelian surfaces admitting an embedding $G_\sigma\subseteq\Aut_0(A)$.
\end{prop}

\begin{proof}
By \Ref{cor}{evenrank}, we have to classify even hyperbolic $p$-elementary lattices of even rank. Since the invariant lattice and its orthogonal have isomorphic discriminant groups, they have the same length $a$, so it needs to satisfy both $a\leq r$ and $a\leq 6-r$. In the case $p=2$, we get a list of possible invariants $(r,a,\delta)$ from \Ref{thm}{2el} and each such triple determines a unique lattice. In the case $p=3$, we use \Ref{thm}{pel} in the same way to get the list above. For $p=5$, \Ref{thm}{pel} gives the couples $(2,0)$, $(2,1)$ and $(2,2)$ for $(r,a)$ and \Ref{prop}{BCMS14} eliminates the cases $a=0$ and $2$. The calculation of the dimensions then follow from \Ref{sec}{modsp}.
\end{proof}

\begin{rmk}
Note that all invariant lattices in the table admit a unique embedding into $U^{\oplus 3}$ as can be checked by a case by case analysis.
\end{rmk}

In the following subsections we study all Fujiki's families from a lattice-theoretic point of view. We give also the fixed loci for elements of $G_\sigma\setminus\{\pm\id\}$ in each case.

Note that by \Ref{rmk}{general} the general element in each family satisfies $\NS(A)=T(G_\sigma)$.

\subsection{Fujiki's families with {\boldmath$p=2$}}
\begin{ex}\label{ex:nonsymp2a} {\boldmath$T(G_\sigma)\cong U\colon G_\sigma=(\ZZ/2\ZZ)^{\times2}$} \vspace{5pt} \\
Let $E$ and $E'$ be arbitrary elliptic curves. The automorphism $\sigma:=(\id_E,-\id_{E'})$ on $A:=E\times E'$ is non-symplectic of order two. We have
\[\NS(A)\cong T(G_\sigma)\cong \langle E\times\{0\},\{0\}\times E'\rangle\cong U.\]
Since the two factors $E$ and $E'$ both have one deformation parameter, we see that in this way we obtain a two dimensional family of non-symplectic involutions.

The fixed locus of $\sigma$ is isomorphic to $4$ copies of $E$ and the fixed locus of $-\sigma$ is isomorphic to $4$ copies of $E'$. 
\end{ex}
\begin{ex}\label{ex:nonsymp2b} {\boldmath$T(G_\sigma)\cong U(2)\colon G_\sigma=(\ZZ/2\ZZ)^{\times2}$} \vspace{5pt} \\
Consider $A=E\times E'$ as before. Now choose two points of order two $x_0\in E$ and $x_0'\in E'$. Define $A':=A/\langle (x_0,x_0')\rangle$, this is again an abelian surface. Denote by $\bar{E}$ and $\bar{E}'$ the images of $E\times\{0\}$ and $\{0\}\times E'$, respectively. The involution $\sigma$ on $A$ of \Ref{ex}{nonsymp2a} fixes $(x_0,x_0')$ and thus descends to a non-symplectic involution $\bar{\sigma}$ on $A'$. Again we calculate
\[\NS(A')\cong T(G_{\bar{\sigma}})\cong \langle\bar{E},\bar{E}'\rangle\cong U(2).\]

The fixed locus of $\bar{\sigma}$ is isomorphic to $2$ copies of $E$ and the fixed locus of $-\bar{\sigma}$ is isomorphic to $2$ copies of $E'$.
\end{ex}
\begin{ex} {\boldmath$T(G_\sigma)\cong \langle2\rangle\oplus\langle -2\rangle\colon G_\sigma=(\ZZ/2\ZZ)^{\times2}$} \vspace{5pt} \\
Continuing in the same fashion as the two examples above, we choose two distinct order two points and define $A''$ to be the quotient of $A=E\times E'$ by the subgroup (isomorphic to $(\Z/2\Z)^{\times2}$) generated by them. Again the automorphism $\sigma$ from \Ref{ex}{nonsymp2a} descends and we obtain our desired family. From this construction, we find fixed loci respectively isomorphic to $E$ and $E'$ for $\sigma$ and $-\sigma$.

An open subfamily of this case can be described by Jacobians of genus two curves: Let $C$ be the genus two curve given by the following hyperelliptic model:
\[y^2=x(x^4+\lambda x^3+\mu x^2+\lambda x +1).\]
The transformation
\[ \varphi\colon (x,y)\mapsto \left(\frac{1}{x},\frac{y}{x^3}\right)\]
gives an involution of the curve which induces an involution $\sigma$ of the Jacobian, which will be our $A''$. Let $p_0\in C$ be one of the two fixed points of $\varphi$. The Abel map
\[C\rightarrow A'',\quad p\mapsto p-p_0\]
embeds $C$ into $A''$ as a $\sigma$-invariant divisor. Furthermore let $E:=C/\varphi$ and denote the quotient map by $\pi$. The curve $E$ is an elliptic curve. The map
\[E\rightarrow A'',\quad x\mapsto \pi^*x-2p_0\]
maps $E$ $2\colon\!1$ to $A''$ with image another $\sigma$-invariant divisor. The intersection number $E.C$ is equal to $2$. Indeed, we have to find points $p\in C$ and $x\in E$ such that $\pi^*x-2p_0 \sim p- p_0$ or equivalently
\[\pi^*x \sim p+ p_0.\]
There are exactly the two possibilities: a) $x=\pi(p_0)$ and $p=p_0$ and b) $x=0$ and $p=\iota(p_0)$. Here $\iota$ denotes the hyperelliptic involution on $C$ and $0\in E$ is the image of the two fixed points of $\varphi\circ\iota$ (which are interchanged by $\iota$). Thus we get the intersection matrix
\[\begin{pmatrix}
2&2\\
2&0
\end{pmatrix},\]
from which we can easily deduce that
\[\NS(A'')\cong T(G_\sigma)\cong\langle2\rangle\oplus\langle -2\rangle.\]

\vspace{10pt}
\begin{rmk}
The connection between the two constructions above is as follows: The two elliptic curves are $E:=C/\varphi$ and $E':=C/(\varphi\circ \iota)$ (where $\iota$ denotes the hyperelliptic involution). Both curves map $2:\!1$ to $A''$ yielding a surjective degree $4$ homomorphism $E\times E'\rightarrow A''$ as in the original construction.
\end{rmk}

 

\begin{rmk}
It is not clear to us at the moment if the above construction admits an inverse, i.e.\ if every quotient $A''$ of $A=E\times E'$ by two points of order two can be obtained as the Jacobian of a genus two curve as described above.
\end{rmk}

\end{ex}
\begin{ex} \label{ex:nonsymp2d} {\boldmath$T(G_\sigma)\cong U\oplus\langle -2\rangle^{\oplus2}\colon G_\sigma=\ZZ/4\ZZ$} \vspace{5pt} \\ 
Let $E_4$ denote the elliptic curve $\C/\langle(1,i)\rangle$. It carries an order four automorphism induced by the multiplication with $i$, which we will denote by $i$, too. Let $x_0$ denote the image of $(1/2(1+i))$. It is (besides the origin) the only fixed point of this automorphism. Now, define $A:=E_4\times E_4$ and $\sigma:=(i,i)$. We have
\[\NS(A)\cong T(G_\sigma)\cong \langle E_4\times \{0\},\{0\}\times E_4,\Delta,\Delta_i\rangle,\]
where $\Delta:=\{(x,x)\mid x\in E_4\}$ denotes the diagonal and we define $\Delta_i:=\{(x,ix)\mid x\in E_4\}$. Computing the intersection product we end up with the following matrix:
\[\begin{pmatrix}
0&1&1&1\\
1&0&1&1\\
1&1&0&2\\
1&1&2&0
\end{pmatrix}.\]
We compute that the determinant equals $-4$ and that it is primitive. Altogether we end up with an isolated example of a non-symplectic automorphism of $A$ of order four whose square equals $-\id$. In this case we therefore have $G_\sigma=\ZZ/4\ZZ$.

Both $\sigma$ and $-\sigma=\sigma^{-1}$ have the same fixed locus, which consists of $4$ points.
\end{ex}
\subsection{Fujiki's families with {\boldmath$p=3$}}
\begin{ex} {\boldmath$T(G_\sigma)\cong U\colon G_\sigma=\ZZ/2\ZZ\times\ZZ/3\ZZ$} \vspace{5pt} \\
Let $E$ be an arbitrary elliptic curve and $E_6$ the elliptic curve $\C/\langle(1,\zeta_6)\rangle$ which comes with an order three automorphism (multiplication by $\zeta_3$) which we will denote by $\zeta_3$. On $A:=E\times E_6$ define the automorphism $\sigma:=(\id_E,\zeta_3)$. It is non-symplectic of order three. We have
\[\NS(A)\cong T(G_\sigma)\cong \langle E\times\{0\},\{0\}\times E_6\rangle\cong U.\]
Deforming the first factor, this yields a one dimensional family of non-symplectic automorphisms of order three.\\

Both $\sigma$ and $\sigma^2=\sigma^{-1}$ have the same fixed locus, which is isomorphic to $3$ copies of $E$, while the fixed locus of $-\sigma$ and of $-\sigma^2$ consists of $4$ points.
\end{ex}
\begin{ex} {\boldmath$T(G_\sigma)\cong U(3)\colon G_\sigma=\ZZ/2\ZZ\times\ZZ/3\ZZ$} \vspace{5pt} \\
Consider $A=E\times E_6$ as before. Choose a point $x_0\in E[3]$ of order three and let $x_0'\in E_6$ be the image of $\frac{1}{2}+i\frac{1}{2\sqrt{3}}$. It is a point of order three of $E_6$ and it is fixed by $\zeta_3$. Thus we can define $A':=A/\langle (x_0,x_0')\rangle$ and we see that it inherits a non-symplectic automorphism $\bar{\sigma}$ of order three. We can easily derive
\[\NS(A')\cong T(G_{\bar{\sigma}})\cong \langle \bar{E},\bar{E_6}\rangle\cong U(3).\]

Again, $\bar{\sigma}$ and $\bar{\sigma}^2=\bar{\sigma}^{-1}$ both have the same fixed locus, which is isomorphic $E$, while the fixed locus of $-\bar{\sigma}$ and of $-\bar{\sigma}^2$ consists of $4$ points.
\end{ex}
\begin{ex} {\boldmath$T(G_\sigma)\cong U\oplus A_2\colon G_\sigma=\ZZ/2\ZZ\times\ZZ/3\ZZ$} \vspace{5pt} \\
This example follows closely \Ref{ex}{nonsymp2d}. Set $A:=E_6\times E_6$ and $\sigma:=(\zeta_3,\zeta_3).$ This is a non-symplectic automorphism of order three and we compute
\[\NS(A)\cong T(G_\sigma)\cong \langle E_6\times \{0\},\{0\}\times E_6,\Delta,\Delta_{\zeta_3}\rangle,\]
where again $\Delta$ denotes the diagonal and we set $\Delta_{\zeta_3}:=\{(x,\zeta_3x)\mid x\in E_6\}$. Computing intersection we get the following matrix:
\[\begin{pmatrix}
0&1&1&1\\
1&0&1&1\\
1&1&0&3\\
1&1&3&0
\end{pmatrix},\]
which can easily be proven to be congruent to $U\oplus A_2(-1)$.

Both $\sigma$ and $\sigma^2=\sigma^{-1}$ have the same fixed locus, which is made of $9$ points, while the fixed locus of $-\sigma$ and $-\sigma^2$ consists of a single point.
\end{ex}
\subsection{Fujiki's family with {\boldmath$p=5$}}
\begin{ex} {\boldmath$T(G_\sigma)\cong H_5\colon G_\sigma=\ZZ/2\ZZ\times\ZZ/5\ZZ$} \vspace{5pt} \\
Let $A=\C^2/\Lambda$ be the simple abelian surface with $\Lambda:=\lat{(1,1),(\zeta_5,\zeta_5^2),(\zeta_5^2,\zeta_5^4),(\zeta_5^3,\zeta_5)}$, where $\zeta_5$ is a fifth root of unity.
Then $H^1(A,\mathbb{Z})$ is generated by $e_1,e_2,e_3$ and $e_4$, which are the couples generating $\Lambda$ above. Let $\sigma$ be the automorphism defined as multiplication by $(\zeta_5,\zeta_5^2)$ on $A$.
A direct computation shows that the invariant lattice of inside $H^2(A,\mathbb{Z})$ is generated by 
\begin{align*}
2 e_1\wedge e_3+ 2 e_2\wedge e_4+2 e_2\wedge e_3-e_3\wedge e_4-e_1\wedge e_4-e_1\wedge e_2 & \quad\text{and}\\
e_1\wedge e_2 +e_1\wedge e_3 +e_1\wedge e_4+ e_2\wedge e_3+e_2\wedge e_4+e_3\wedge e_4, &
\end{align*}
which has intersection matrix given by $\left( \begin{array}{cc} 10 & 5\\ 5 & 2\end{array}\right)$, therefore it is isometric to the lattice $H_5$ in the table.

The automorphisms $\sigma^k$ for $k=1,2,3,4$ have the same fixed locus, which is made of $5$ points, while the fixed locus of $-\sigma^k$ for $k=1,2,3,4$ consists of a single point.

\begin{rmk}
There is a more geometric construction of this example, again as a Jacobian: Let $C$ be the genus two curve given by the following hyperelliptic model:
\[y^2=x^5-1.\]

The coordinate transformation given by $(x,y)\mapsto(\zeta_5x,y)$ gives an order five automorphism of $C$. Let $A:=\Pic^0(C)$ be its jacobian. By pulling back along the above transformation we get an induced automorphism $\sigma$ of order five on $A$. Let $p_0\in C$ be a fixed point of the coordinate transformation. The abel map
\[C\rightarrow A,\quad p\mapsto [p-p_0] \]
embeds $C$ into $A$ as an invariant divisor. In total there are three fixed points in $C$, say $p_0$, $p_1$ and $p_2$. Assume that $p_0$ is fixed by the hyperelliptic involution of $C$. We thus obtain the five fixed points of $A$ as $[0]$, $[p_1-p_0]$, $[p_2-p_0]$, $[p_0+p_1-2p_2]$ and $[p_0+p_2-2p_1]$.
\end{rmk}

\end{ex}

\subsection{Fujiki's families and duality}\label{ssec:duality}

Let us briefly analyse the dual families of the above examples. Here we must distinguish two cases, namely when the tori are self-dual or not. In the first case we only have to check whether the dual automorphism is the same or not. Let $(A,H)$ be a principally polarised abelian variety with a non-symplectic automorphism $\sigma$ which is different from $-\id_A$. The general such $A$ will satisfy $\sigma^*H=H$ and it is then a straightforward calculation that the dual automorphism to $\sigma$ (i.e.\ pullback along $\sigma$) is just given by conjugating $\sigma$ with the isomorphism $\varphi_H\colon A\rightarrow A^*$ (i.e.\ $\sigma^*=\varphi_H\circ \sigma\circ \varphi_H^{-1}$). We can thus state that the dual family is identical to the original one.

If we look at families where the general member is a torus $A$ with automorphism $\sigma$ which is not principally polarised, the moduli space of such pairs $(A,\sigma)$ has two distinct connected components, which are interchanged by applying duality. There might very well be subloci of these two components which correspond to pairs $(A,\sigma)$ with $A$ principally polarised, thus for the tori duality acts as identity. But in this case, the polarisation is not invariant under $\sigma$ and thus a direct computation shows that the dual automorphism is distinct from $\sigma$.


\section{Symplectic automorphisms of prime order}
In this section we will give a classification of families of prime order symplectic automorphisms of abelian surfaces.

Let $A$ be a complex torus admitting a sympletic automorphism with an action of prime order $p$ on $H^2(A,\Z)$. We saw in the preliminaries that $p$ equals $2$, $3$ or $5$. To further reduce this list, we use the following
\begin{lem}[{\cite[Lemma 3.3]{Fuj88}}]
Let $A$ be a complex torus and let $\sigma$ be a symplectic automorphism on it. Then $\sigma$ has order $1,2,3,4$ or $6$.
\end{lem}
The only symplectic automorphism of order 2 is $-\id$, which exists on all complex tori. The only interesting cases are automorphisms of orders 3 and 4 (6 being related to 3 by composing with $-\id$), which are classified by \cite[Proposition 3.7]{Fuj88} and described below.

Before going into details with the examples, let us take a look at the quotient surfaces: $A/\sigma$ is simply connected and has singularities only on points which have a non-trivial stabiliser. These points are a $A_{n-1}$ type singularity, where $n$ is the order of the stabiliser. A direct way to see this is by taking local coordinates and writing the action of the stabiliser on the tangent space around the point. This matrix has order $n$ and lies in $SL_2(\mathbb{C})$, hence it is $\left(\begin{array}{cc} \zeta_n & 0\\ 0& \zeta_n^{-1}\end{array}\right)$, where $\zeta_n$ is a $n$-th root of unity.
Therefore, a minimal resolution of singularities $\pi\colon S\rightarrow A/\sigma$ is a $K3$ surface. We have a commutative diagram
\[\xymatrix{ \overline{A} \ar[r]^{\overline{q}} \ar[d]_{\overline{\pi}} & S \ar[d]_\pi \\
A \ar[r]^q & A/\sigma,
}\]
where we denote the quotient $A\rightarrow A/\sigma$ by $q$ and $\overline{A}$ is the fibre product. 
Let $\alpha:=\overline{q}_{!}\overline{\pi}^*\colon T(A)\rightarrow T(S)$ be the induced map on the transcendental lattices.
\begin{lem}\label{lem:transc}
Let $A,$ $S$ and $\alpha$ be as above and let $n=ord(\sigma)$. Then for all $x,y\in T(A)$ we have $(\alpha(x),\alpha(y))=n(x,y)$ and $\alpha$ is an isomorphism.

\end{lem}  
\begin{proof}
The case $n=2$ is \cite[Proposition VIII.5.1]{BHPV04}, and the same proof works also for $n>2$. Indeed, every element $z\in T(\overline{A})$ is invariant under the covering morphisms for $\overline{p}\colon\overline{A}\rightarrow S$, so we can write $z=\overline{q}^*w$ for some $w\in T(S)$ and the projection formula shows that $\overline{q}^*(\overline{q}_{!}z)=nz$. thus $(\overline{q}_{!}z_1,\overline{q}_{!}z_2)=n(z_1,z_2)$, and $\overline{q}_{!}$ and $\alpha$ multiply intersection numbers by $n$. The map $\alpha$ is a map between irreducible Hodge structures, hence it is either zero or an isomorphism. As it sends the symplectic form of $A$ to the symplectic form of $S$, it must be non-zero.  
\end{proof}

\begin{ex}{\boldmath$T(G_\sigma)\cong U\oplus A_1^2\colon G_\sigma=(\ZZ/4\ZZ)$} \vspace{5pt} \\
We set $A=\C^2/\Lambda$, where $\Lambda=\langle(1,0),(0,1),(x,-y),(y,x)\rangle$, with $(x,y)\in\C^2\setminus\R^2$. Then $\sigma=\begin{pmatrix} 0&-1 \\ 1&0 \end{pmatrix}$ is a symplectic automorphism of order 4 on $A$, and every order 4 symplectic automorphism on a complex torus is of this form.
The isometry $\varphi$ induced by $\sigma$ on $H^2(A,\Z)$ has order $2$ and coinvariant lattice isometric to $A_1(-1)^2$. Let us analyse the fixed locus of such an automorphism: in the above family we can degenerate to $y=0$ and we obtain a one dimensional subfamily of abelian surfaces of the form $E_x\times E_x$. Here the automorphism exchanges the two curves and changes sign on one of them, therefore the set of fixed points consists of four points of the form $(a,a)$, where $a\in E_x$ has order two. There are also twelve points with a nontrivial stabiliser and they are of the form $(a,b)$, where $a,b\in E_x$ are distinct and of order two. Consider the quotient $A/\sigma$ for general $A$ with a symplectic automorphism $\sigma$ as above. The surface $A/\sigma$ is simply connected and has four $A_3$ singularities and six of type $A_1$. Its minimal resolution of singularities $S$ is a $K3$ surface with Picard lattice generated (over $\mathbb{Q}$) by the
 $-2$ curves over the singularities. We can apply \Ref{lem}{transc} to compute its transcendental lattice, which is $T(S)=T(G_\sigma)(4)\cong U(4)\oplus A_1(4)^2$. The Picard lattice is then the orthogonal complement of this lattice inside the $K3$ lattice. 
\end{ex}

\begin{ex}{\boldmath$T(G_\sigma)\cong U\oplus A_2 \colon G_\sigma=(\ZZ/6\ZZ)$} \vspace{5pt} \\
We set $A=\C^2/\Lambda$, where $\Lambda=\langle(1,0),(0,1),(x,y),(-y,x-y)\rangle$, with $(x,y)\in\C^2\setminus\R^2$. Then $\sigma=\begin{pmatrix} 0&-1 \\ 1& -1 \end{pmatrix}$ is a symplectic automorphism of order 3 on $A$, and every order 3 symplectic automorphism on a complex torus is of this form.
The isometry $\varphi$ induced by $\sigma$ on $H^2(A,\Z)$ has coinvariant lattice isometric to $A_2(-1)$. Let us analyse the fixed locus of such an automorphism: in the above family we can choose $y=0$ and we obtain a one dimensional subfamily of abelian surfaces of the form $E_x\times E_x$. The set of fixed points consists of nine points of the form $(a,2a)$, where $a\in E_x$ has order three. If as before we consider a $K3$ surface which is a minimal resolution of the general element of this family, we have a $K3$ surface with $T(S)\cong U(3)\oplus A_2(3)$ by \Ref{lem}{transc}.  
\end{ex}
\section{Moduli spaces}\label{sec:modsp}
In this section we discuss several aspects of moduli spaces of complex tori with an automorphism.

\subsection{Non-symplectic case}
We follow closely \cite[Sect.\ 9]{AST11}. We want to study moduli spaces of pairs $(A,G)$ consisting of an abelian surface $A$ and a non-symplectic automorphism group of prime order $G$. First we will introduce the correct notion of \em pair \em for our purposes:

\begin{defn}
Let $H\subset O(U^3)$ be a cyclic subgroup of prime order $p$ with hyperbolic invariant lattice and let $[H]$ be its conjugacy class. An \em $[H]$-polarised \em abelian surface is a pair $(A,G)$ consisting of an abelian surface $A$ and a non-symplectic subgroup $G\subseteq \aut_0(A)$ such that for some marking $\phi$ we have:
\[G=\phi^*H\]
An \em isomorphism of $[H]$-polarised surfaces \em is an isomorphism $f\colon A\rightarrow A'$ commuting with the $G$-actions.
\end{defn}

We can construct a moduli space of polarised surfaces as follows: Let $S(G):=T(G)^\perp$ be the coinvariant lattice of $G$. We choose a generator $\sigma$ of $H$, fix $\zeta_p$ a primitive $p$th root of unity and set
\[ V^H:= \{x\in U^3\otimes \CC\mid \sigma(x)=\zeta_px\}\subset S(G)\otimes \CC\]
and define our period domain as
\[ D^H:= \{w\in \PP(V^H)\mid (w,\bar{w})>0, (w,w)=0\}.\]
The domain $D^H$ is of dimension $4-r$ if $p=2$ and $\frac{6-r}{p-1}-1$ for $p=3,5$ and admits a natural action by the discrete group $\Gamma^H=\{\gamma\in O(U^3)\mid \gamma\circ H=H\circ\gamma$\}.

We follow \cite[Thm.\ 9.1]{AST11} to prove the following statement:

\begin{prop}
The quotient space $\Gamma^H\backslash D^H$ is a moduli space parametrizing isomorphism classes of $[H]$-polarised pairs $(A,G)$.
\end{prop}
\begin{proof}
If $(A,G)$ is an $[H]$-polarised pair together with a suitable marking $\phi$ as above, then its period clearly lies in $D^H$. Furthermore the marking $\phi$ is unique up to the action of $\Gamma^H$. Thus we obtain a well defined point in the quotient $\Gamma^H\backslash D^H$.

To prove the surjectivity of the correspondence, let $w$ be a point in $\Gamma^H\backslash D^H$. By the surjectivity of the period map (\Ref{thm}{shio3}) there exists a 2-torus $A$ with a compatible marking $\phi$ such that the period $[\phi(\omega_A)]$ equals $w$. Set $\widetilde{G}:=\phi^*H$, then by \Ref{thm}{shio} there exists a group of non-symplectic automorphisms $G$, such that $\nu(G)=\widetilde{G}$.
\end{proof}

\begin{rmk}\label{rmk:general}
It follows from the above construction that for a fixed polarisation type $[H]$ the general element of the family of $[H]$-polarised tori has Picard group isomorphic to $T(G)$, because there are no more integral classes orthogonal to a generic period in $D^H$.

Also note that the above description should be understood as a theoretical concept rather than an effective tool to understand the geometry of the moduli spaces. In fact, in the examples from \Ref{sec}{nonsymp} it is easy to determine the geometric structure of the families directly from the construction of the examples.
\end{rmk}

\subsection{Symplectic case}
In the symplectic case, the period of the surface we are considering is invariant by the automorphism group $G$. Therefore, we can consider a restricted period domain $\Omega_{T(G)}$, given by all periods contained in $\mathbb{P}(T(G)\otimes\mathbb{C})$. Let $\mathcal{M}_G$ be the preimage by the period map of this space. With an appropriate choice of marking, all complex tori with an action of $G$ lie in $\mathcal{M}_G$. Let us restrict ourselves to a single connected component of $\mathcal{M}_G$, as the other can be recovered by duality. The generic point of $\mathcal{M}_G$ corresponds to a marked abelian surface $(A,f)$ such that $T(G)=T_A$ and $S(G)=NS(A)$. Moreover $G$ acts by monodromy Hodge isometries, hence, by \Ref{thm}{shio2}, $G\subset Aut_0(A)$.

To prove that this holds true also for every point in $\mathcal{M}_G$, we need to introduce two pieces of notation. Given any period domain, the set of periods contained in a fixed positive three space $P$ is called a twistor line, which is generic if it contains the period of a generic Hodge structure. Given an invariant \kahl class $\omega$ on a complex torus $A$ and the \hk metric $g$ associated with it, we obtain a family of complex tori given by all the \kahl metrics corresponding to $g$, which is a $\mathbb{P}^1$ and is called a twistor family. The periods corresponding to a twistor family form a twistor line, where the space $P$ is generated by the image of the symplectic form, its conjugate and the \kahl class. In the general case of \hk manifolds, a twistor line does not lift to a twistor family. However, in the specific setting of abelian surfaces they all lift:
\begin{prop}
Let $L$ be a twistor line in the period domain of abelian surfaces. Then there exists a complex torus $A$ with a \hk metric $g$ and a marking $f$ which sends the twistor family of $g$ to the twistor line $L$. 
\end{prop}
\begin{proof}
Let us take a point $p\in L$ and let $(A,f)$ be a marked complex torus in the preimage $\mathcal{P}^{-1}(p)$ of the period map. Let $P$ be the positive space corresponding to the twistor line $L$. Then $f^{-1}(P)$ intersects the positive cone of $A$ in a half line spanned by $\omega$. The class $\omega$ is K\"ahler, as all positive classes are K\"ahler. Therefore there exists a \hk metric $g$ corresponding to this \kahl class and the twistor family associated to it is sent to $L$ by the period map. 
\end{proof}
The lattice $T(G)$ has signature $(3,a)$, therefore the period domain $\Omega_G$ is covered by twistor lines as proven in \cite[Exp. VIII]{beau_tor}, moreover we can cover it with generic twistor lines. By the above propositions, these lines lift to twistor families which actually lie in $\mathcal{M}_G$ and the \hk metric is $G$ invariant. This defines an action of $G$ on all the elements of $\mathcal{M}_G$, which is therefore the moduli space of complex tori with a symplectic action of $G$ and all its connected components are rationally connected.


\end{document}